\documentclass[11pt,leqno]{amsart}

\usepackage{graphicx}
\usepackage{epstopdf,epsfig}
\usepackage[colorlinks]{hyperref}
\usepackage{csquotes}
\hypersetup{				
	colorlinks,
	linkcolor={blue!80!black},
	citecolor={blue!80!black},
	urlcolor={blue!80!black}
}

\usepackage{todonotes}
\usepackage{xcolor}

\usepackage{amsmath,amsthm,amsfonts,amssymb,latexsym,amscd,enumerate,mathtools}

\usepackage{palatino}
\usepackage[mathcal]{euler}

\usepackage{xy}
\xyoption{all}

\theoremstyle{remark}

\theoremstyle{plain}

\newcounter{theoremintro} 

\newtheorem{introtheorem}[theoremintro]{Theorem}
\newtheorem{introcorollary}[theoremintro]{Corollary}

\newtheorem*{definition*}{Definition} 
\newtheorem*{theorem*}{Theorem} 
\newtheorem*{lemma*}{Lemma}
\newtheorem*{corollary*}{Corollary}

\newtheorem{theorem}[subsection]{Theorem} 
\newtheorem{lemma}[subsection]{Lemma}
\newtheorem{corollary}[subsection]{Corollary}
\newtheorem{proposition}[subsection]{Proposition}

\theoremstyle{definition}
\newtheorem{definition}[subsection]{Definition}

\theoremstyle{remark}

\newtheorem{remark}[subsection]{Remark}

\numberwithin{equation}{section}

\newcommand{\GL}{\mathrm{GL}}

\newcommand{\N}{\mathbb{N}}
\newcommand{\Z}{\mathbb{Z}} 
\newcommand{\Q}{\mathbb{Q}}

\newcommand{\C}{\mathbb{C}}

\newcommand{\bZ}{\mathbb{Z}}
\newcommand{\bQ}{\mathbb{Q}}

\newcommand{\fm}{\mathfrak{m}}
\newcommand{\fn}{\mathfrak{n}}

\newcommand{\class}{{\mathrm{class}}}
\DeclareMathOperator{\End}{End}
\DeclareMathOperator{\id}{id}

\begin{document} 
\title[Bernoulli shifts with approximately inner flip]{Equivariant $KK$-theory of Bernoulli shifts on $C^*$-algebras with approximately inner flip}

		\author[J. Kranz]{Julian Kranz}
	\address{J.Kranz: Chair of Data Science,		
		Institut für Wirtschaftsinformatik,		
		Leonardo Campus 3, 48149 M\"unster, Germany.}
	\email{julian.kranz@uni-muenster.de}
	\urladdr{https://sites.google.com/view/juliankranz/}

\author[S. Nishikawa]{Shintaro Nishikawa} 
\address{S. Nishikawa: School of Mathematical Sciences, University of Southampton, University Road, Southampton, SO17 1BJ, UK}
\email{s.nishikawa@soton.ac.uk}
\urladdr{https://sites.google.com/view/snishikawa/}




\maketitle

\begin{abstract}
Building on Enders--Schemeitat--Tikuisis' classification, we show that a separable $C^*$-algebra $A$ with approximately inner flip in the UCT class is $K$-theoretically self-absorbing if and only if for every finite group $G$, the Bernoulli shift on $A^{\otimes G}$ is $KK^G$-equivalent to the trivial action.
This in particular applies to UHF-algebras of infinite type and computes the $K$-theory of the associated crossed product. 
Along the way, we obtain an alternative proof of Hirshberg--Winter's result that the Bernoulli shift of $G$ on a UHF-algebra of infinite type absorbs the trivial action up to conjugacy. 
For more general amenable groups $G$, we develop $K$-theory formulas for Bernoulli shifts on UHF-absorbing $C^*$-algebras, and establish $KK^G$-triviality for Bernoulli shifts on strongly self-absorbing $C^*$-algebras satisfying the UCT. 
%
%
%
\end{abstract}

\tableofcontents
\section{Introduction}
In topological dynamics, a very fertile class of examples is given by Bernoulli shifts, that is, by the shift action of a group $G$ on the product $X^G\coloneqq \prod_G X$ of $G$-many copies of a given compact space $X$. 
When the space $X$ is moreover totally disconnected, the $K$-theory of the crossed product $C\left(X^G\right)\rtimes_r G$ can be computed in many cases \cite{CEL}. 
These computations and the techniques appearing in them are not only of intrinsic interest, but they make possible the computation of the $K$-theory of $C^*$-algebras associated to large classes of (inverse) semigroups, wreath products, and many more examples \cite{CEL,XinLi,XinLiPartial}.

The non-commutative version of the Bernoulli shift is the shift action of a group $G$ on the tensor product $A^{\otimes G}\coloneqq \bigotimes_{g\in G} A$ for a given unital $C^*$-algebra $A$. These \emph{non-commutative Bernoulli shifts} have a long history in operator algebras originating from non-commutative entropy and the classification of group actions \cite{CS,Voiculescu1995,Popa,Szabo2019,GardellaLupini}.

The simplest non-commutative analogue of a totally disconnected space is a \emph{UHF-algebra}, that is, a (possibly infinite) tensor product of matrix algebras $M_\mathfrak n=\bigotimes_{p} M_p^{\otimes \mathfrak n_p}$ for a \emph{supernatural number} $\mathfrak n = \prod_p p^{n_p}$ with $n_p\in \N\cup \{\infty\}$ for all primes $p$. 
A key feature of UHF algebras is that they have \emph{approximately inner flip} \cite{ER} in the sense that the flip map \[\sigma_{A,A}\colon A\otimes A\to A\otimes A,\quad a\otimes b \mapsto b\otimes a\]
is a point-norm limit of inner automorphisms. 

Our main result computes the $K$-theory of the associated crossed product in the case that $G$ is finite.
To state it, we call a supernatural number $\mathfrak n$ as above of \emph{infinite type} if $n_p\in \{0,\infty\}$ for all $p$. 
For any supernatural number $\mathfrak n$, we write
\[\Q_\mathfrak n \coloneqq \left\{\frac a b  \,\middle\vert\, a\in \Z, b\in \N \text{ divides } \mathfrak n\right\}\subset\Q.\]
\begin{introtheorem}[Theorem \ref{thm_GKK}] \label{thmD} Let $A$ be a separable $C^*$-algebra satisfying the UCT \cite{Rosenberg1987}. The following are equivalent:
	\begin{enumerate}
		\item\label{item-classif-intro} $A$ is $KK$-equivalent to a unital, simple, separable, nuclear, $\mathcal Z$-stable $C^*$-algebra $A$ with approximately inner flip such that $A\otimes A\cong A$;
		\item\label{item-K-absorbing-intro} The flip map $\sigma_{A,A}$ is equal to the identity in $KK(A\otimes A,A\otimes A)$ and we have an isomorphism $K_*(A)\cong K_*(A\otimes A)$; 
		\item\label{item-C2-trivial-intro} The flip action on $A^{\otimes C_2}$ is $KK^{C_2}$-equivalent to the trivial action on $A\otimes A$;
		\item\label{item-G-trivial-intro} For any finite group $G$, and for any finite $G$-set $Z$, $A^{\otimes Z}$ equipped with the Bernoulli shift $G$-action is $KK^G$-equivalent to $A$ equipped with the trivial $G$-action;
		\item\label{item-K-theory-intro} As a graded abelian group, $K_0(A)\oplus K_1(A)$ is isomorphic to either $0\oplus \Q_\fm/\Z$ or $\Q_\fn\oplus \Q_\fm/\Z$ for supernatural numbers $\fm, \fn$ of infinite type such that $\fm$ divides $\fn$. 
	\end{enumerate}
	In particular, for any $G$ and $Z$ as in \eqref{item-G-trivial-intro}, we have an isomorphism of $R_\C(G)$-modules
	\[K_*(A^{\otimes Z}\rtimes_r G)\cong K_*(A)\otimes_\Z R_\C(G),\]
	where $R_\C(G)\coloneqq K_0(C^*(G))\cong \Z[\hat G]$ denotes the complex representation ring. 
\end{introtheorem}
Our proof of Theorem \ref{thmD} builds on Enders--Schemeitat--Tikuisis' classification \cite{Tikuisis,TikuisisCorrigendum} of $C^*$-algebras with approximately inner flip satisfying the assumptions of the Elliott classification programme\footnote{We refer to \cite{WinterICM,White2023} and the references therein for an overview of the Elliott programme.}. 
A key step of the proof first establishes the case of UHF-algebras of infinite type using a representation theoretic argument. 
The main technical ingredient for combining the UHF-case and Enders--Schemeitat--Tikuisis' classification is a certain filtration of the Bernoulli shift action by invariant ideals that was introduced by Izumi \cite{Izumi} and later used in \cite{CEKN, Bunke2023} (see Proposition \ref{prop_Izumi_filtration}).
We prove the UHF case in slightly higher generality than that of Theorem \ref{thmD}:
\begin{introtheorem}[Theorem \ref{thm-kk-equivalence}] \label{thmA}
	Let $G$ be a finite group, let $Z$ be a countable $G$-set and let $M_\mathfrak n$ be a UHF-algebra of infinite type. Then $M_\mathfrak n$ is $KK^G$-equivalent to $M_\mathfrak n^{\otimes Z}$ where we equip $M_\mathfrak n$ with the trivial $G$-action and $M_{\mathfrak n}^{\otimes Z}$ with the Bernoulli shift. In particular, we have
		\[K_*\left(M_\mathfrak n^{\otimes Z}\rtimes G\right)\cong K_*(C^*(G)\otimes M_\mathfrak n)\cong K_*(C^*(G))[1/\mathfrak n].\]
\end{introtheorem}
The proof of Theorem \ref{thmA} relies on a representation theoretic argument about invertibility of a certain element in the representation ring $R_\C(G)$ after inverting sufficiently many primes (see Proposition \ref{prop-KK-elements}). A byproduct of the proof is that the Bernoulli shift absorbs the trivial action not only in $KK$-theory, but up to conjugacy. This reproves a result by Hirshberg--Winter (see \cite[Corollary 3.2]{HirshbergWinter} combined with \cite[Theorem 2.6]{Szabo2018}).
\begin{introtheorem}[Hirshberg--Winter, see Theorem \ref{thm-absorb}] \label{thmB}
	With the notation as in Theorem \ref{thmA}, there is a $G$-equivariant isomorphism 
		\[M_\mathfrak n^{\otimes Z} \cong M_\mathfrak n \otimes M_\mathfrak n^{\otimes Z}.\]
\end{introtheorem}
One immediate consequence of Theorem \ref{thmA} and \cite[Theorem 3.13]{Izumi2} is that the Bernoulli shift $G\curvearrowright M_\mathfrak n^{\otimes Z}$ as above does not have the Rokhlin property (see Corollary \ref{cor-Rokhlin}).
Beyond finite group actions, Theorem \ref{thmA} also has consequences for infinite groups satisfying the Baum--Connes conjecture with coefficients \cite{BCH}. 
\begin{introcorollary}[Corollary \ref{cor-BCC}]\label{corB}
	Let $G$ be a countable discrete group satisfying the Baum--Connes conjecture with coefficients, let $Z$ be a $G$-set, let $A$ be a $G$-$C^*$-algebra and let $M_\mathfrak n$ be a UHF-algebra. Assume that $Z$ is infinite or that $\mathfrak n$ is of infinite type. Then the inclusion $A\to A\otimes M_\mathfrak n^{\otimes Z}$ induces an isomorphism 
		\[K_*\left(A\rtimes_r G\right)[1/\mathfrak n]\cong K_*\left(\left(A\otimes M_\mathfrak n^{\otimes Z}\right)\rtimes_r G\right).\]
	In particular, the right hand side is a $\Z[1/\mathfrak n]$-module.
\end{introcorollary}
Corollary \ref{corB} is particularly useful for analyzing the K-theory of Bernoulli shift actions of infinite groups, and it has been used in the paper \cite{CEKN} together with S. Chakraborty and S. Echterhoff to compute the $K$-theory of many more general Bernoulli shifts. A similar strategy has been used in the context of Farrell--Jones conjecture to compute the algebraic $K$-theory of wreath products \cite{Kranz2024}.

Another consequence of Theorem \ref{thmA} is that the Bernoulli shift of a countable amenable group $G$ on a strongly self-absorbing (in the sense of \cite{TW}) $C^*$-algebra $\mathcal D$ satisfying the UCT is $KK^G$-equivalent to the trivial $G$-action on $\mathcal D$ (see Corollary \ref{cor-ssa}; for $\mathcal D=\mathcal O_\infty$, this is \cite[Corollary 6.9]{Szabo}). 
\subsection*{Acknowledgements}
This research was partly supported by the Deutsche Forschungsgemeinschaft (DFG, German Research Foundation) under Germany's Excellence Strategy EXC 2044--390685587, Mathematics M\"unster: Dynamics--Geometry--Structure.
We acknowledge with appreciation helpful correspondence with Sayan Chakraborty, Siegfried Echterhoff, Jamie Gabe, Eusebio Gardella, Nigel Higson, Masaki Izumi, G\'abor Szab\'o, Aaron Tikuisis and David Vogan. We would like to thank the referee for helpful suggestions.  

\section{$KK$-theory of Bernoulli shifts}\label{sec-KK}
For a finite group $G$, denote by $R_\C(G)$ its representation ring, defined as the Grothendieck group of the monoid of isomorphism classes of finite-dimensional complex representations of $G$ with the direct sum as addition and the tensor product as multiplication. 
The character of a finite-dimensional complex representation $\pi\colon G\to \GL(V_\pi)$ is denoted by
	\[\chi_\pi\colon G\to \C, \,\,\, \chi_\pi(g)\coloneqq \mathrm{tr}\left(V_\pi\xrightarrow{\pi(g)}V_\pi\right),\]
where $\mathrm{tr}$ denotes the (non-normalized) trace. Recall that the map 
 	\[R_\C(G)\to \C_\class(G), \,\,\, \pi\mapsto \chi_\pi\]
is an injective ring homomorphism with values in the algebra $\C_\class(G)$ of conjugation invariant functions on $G$ with pointwise multiplication. There is a natural isomorphism $R_\C(G)\cong KK^G(\C,\C)$. 
We refer to \cite{Serre} for an introduction to representation theory of finite groups and to \cite{Kasparov} for the definition of equivariant $KK$-theory. 
\begin{proposition}\label{prop-KK-elements}
Let $G$ be a finite group, let $k\geq 1$ and let $Z$ be a finite $G$-set. Denote by $\pi_k\colon G\to \GL\left(\ell^2\left(\{1,\dotsc,k\}^Z\right)\right)$ the permutation representation associated to the $G$-set $\{1,\dotsc,k\}^Z$. Then the following hold.
	\begin{enumerate}
		\item There exist $\alpha\in R_\C(G)$ and $r\geq 1$ such that $[\pi_k]^r=k\alpha$.
		\item There exist $\beta\in R_\C(G)$ and $l\geq 1$ such that $[\pi_k]\cdot \beta=k^l$.
\end{enumerate}	
\end{proposition}
\begin{proof}
	By considering the standard basis in $\ell^2\left(\{1,\dotsc,k\}^Z\right)$, it is easy to see that the trace of $\pi_k(g)$ for $g\in G$ is given by the number of $g$-fixed points in $\{1,\dotsc,k\}^Z$, which is the the same as the number of $\langle g\rangle$-invariant functions $Z\to \{1,\dotsc,k\}$. In other words, the character of $\pi_k$ is given by 
		\[\chi_{\pi_k}(g)= k^{|Z/\langle g\rangle |}.\]
	We therefore have
		\[\prod_{g\in G}\left(\chi_{\pi_k}-k^{|Z/\langle g \rangle|}\right)=0 \,\,\text{in} \,\, \C_{\class}(G).\]
	Since the map $\pi\mapsto \chi_\pi$ is injective, we also have 
		\[ \prod_{g\in G}\left([\pi_k] - k^{|Z/\langle g\rangle|}\right)=0 \,\,\text{in} \,\,  R_\C(G).\]
	In particular, there are polynomials $p,q\in \Z[t]$ satisfying 
		\[[\pi_k]^{|G|}=k p([\pi_k]),\quad [\pi_k]\cdot q([\pi_k])=\prod_{g\in G}k^{|Z/\langle g\rangle|},\]
	which proves the proposition.
\end{proof}
\begin{definition}
	Let $Z$ be a set and let $(A_z)_{z\in Z}$ be a collection of unital $C^*$-algebras. The infinite tensor product $\bigotimes_{z\in Z}A_z$ is defined as 
		\[\bigotimes_{z\in Z}A_z\coloneqq \varinjlim_{F} \bigotimes_{z\in F}A_z,\]
	where the inductive limit is taken over all finite subsets $F\subset Z$ ordered by inclusion, with respect to the connecting maps $a\mapsto a\otimes 1$.
	Given a discrete group $G$, a unital $C^*$-algebra $A$ and a $G$-set $Z$, the \emph{Bernoulli shift} of $G$ on $A^{\otimes Z}\coloneqq \bigotimes_Z A$ is the $G$-action induced by permuting the tensor factors according to the $G$-action on $Z$. 
\end{definition}
\begin{definition}
	A \emph{supernatural number} is a formal product $\mathfrak n= \prod_p p^{n_p}$ where $p$ runs over all primes and $n_p\in \{0,\dotsc,\infty\}$. The \emph{UHF-algebra} associated to $\mathfrak n$ is the infinite tensor product 
		\[M_{\mathfrak n}\coloneqq \bigotimes_p M_{p^{n_p}},\]
	with $M_{p^\infty}\coloneqq  M_p^{\otimes \N}$. We call $\mathfrak n$ or $M_{\mathfrak n}$ of \emph{infinite type} if $n_p\in \{0,\infty\}$ for all $p$. We say that $\mathfrak n=\prod p^{n_p}$ divides $\mathfrak m=\prod p^{m_p}$ if $n_p\leq m_p$ for all $p$. 
\end{definition}
\begin{remark}
	Note that the above definition includes natural numbers and matrix algebras as a special case.
\end{remark}
\begin{definition}
	If $M$ is an abelian group, we denote by $M[1/\mathfrak n]$ the inductive limit of the system
		\[M\xrightarrow{\cdot p_1}M\xrightarrow{\cdot p_2}M\xrightarrow{\cdot p_3}\dotsb\]
	where $(p_1,p_2,\dotsc)$ contains each prime dividing $\mathfrak n$ infinitely many times.
\end{definition}
\begin{remark}\label{rem-supernatural}
	 If $\mathfrak q=\prod_p p^{n_p}$ with $n_p\geq 1$ for all $p$, then 
	 	\[M[1/\mathfrak q]\cong M\otimes_\Z \Q.\] 
	 If $k\geq 1$ is a positive integer, then 
		\[\Z[1/k]\cong\left\{\frac m {k^n} \,\middle|\, m\in \Z, n\in \Z_{>0}\right\}\subset\Q.\]
	In general, the group $\Z[1/\mathfrak n]$ is different from the closely related group
		\[\Q_\mathfrak n\coloneqq \left \{\frac m k \,\middle|\, m\in \Z, k\in \Z_{>0} \text{ divides } \mathfrak n\right\},\]
	unless $\mathfrak n$ is of infinite type. 
\end{remark}
As an application of Proposition \ref{prop-KK-elements}, we obtain an alternative proof of \cite[Corollary 3.2]{HirshbergWinter}.
\begin{theorem}[Hirshberg--Winter]\label{thm-absorb}
	Let $G$ be a finite group, let $M_{\mathfrak n}$ be a UHF-algebra and let $Z$ be a $G$-set. Assume that $Z$ is infinite or that $\mathfrak n$ is of infinite type. Equip $M_\mathfrak n$ with the trivial $G$-action and $M_\mathfrak n ^{\otimes Z}$ with the Bernoulli shift. 	
	Then there is an equivariant isomorphism
		\[M_\mathfrak n^{\otimes Z}\otimes M_\mathfrak n\cong M_\mathfrak n^{\otimes Z}.\]
		If $Z$ is infinite, and $m<\infty$, there is an equivariant isomorphism
		\[
		M_ m^{\otimes Z}\otimes M_{ m^\infty} \cong M_ m^{\otimes Z}.
		\]
\end{theorem}
\begin{proof}
	Note that it suffices to prove the statement in the case that $M_\mathfrak n=M_{p^k}$ (or $M_m=M_{p^k}$) for a prime $p$ and $k\in \{0,1,\dotsc,\infty\}$, since the general case follows by taking (possibly infinite) tensor products over all primes. As before, if $Z$ is finite, we denote by $\pi_p$ the permutation representation of $G$ on $V_p\coloneqq \ell^2(\{1,\dotsc,p\}^Z)$, so that $M_p^{\otimes Z}$ is equivariantly isomorphic to $\End(V_p)$. 
	
	Assume first that $k=\infty$. We only need to prove the theorem for (any) one $G$-orbit of $Z$ so we may assume that $Z$ is finite. Let $\alpha\in R_\C(G)$ and $r\geq 1$ be as in Proposition \ref{prop-KK-elements} so that $[\pi_p]^r=p\alpha\in R_\C(G)$. Since $[\pi_p]$ is a non-negative linear combination of irreducible representations of $G$, $\alpha$ has to be the class of a finite-dimensional representation $\pi_\alpha\colon G\to \GL(W_\alpha)$. In particular, we have an equivariant isomorphism $V_p^{\otimes r}\cong \C^p\otimes W_\alpha$. Passing to endomorphisms, we obtain an equivariant isomorphism
		\[\left(M_p^{\otimes Z}\right)^{\otimes r}\cong M_p\otimes \End(W_\alpha)\]
	with the trivial $G$-action on $M_p$. By taking the infinite tensor product we obtain an equivariant isomorphism
		\[M_{p^\infty}^{\otimes Z}\cong M_{p^\infty}\otimes \End(W_\alpha)^{\otimes \N}\cong M_{p^\infty}\otimes M_{p^\infty}\otimes \End(W_\alpha)^{\otimes \N}\cong M_{p^\infty}\otimes M_{p^\infty}^{\otimes Z}.\]
		
	Assume now that $k<\infty$ and that $Z$ is infinite. Then $Z$ contains infinitely many orbits of the same type $G/H$. We may thus assume\footnote{The general case follows by taking tensor products with the remaining factor $M_{\mathfrak n}^{\otimes (Z- \sqcup_\N G/H)}$.} that $Z$ is of the form $Z=\bigsqcup_\N G/H$ for some subgroup $H\subset G$. Then there is an equivariant isomorphism $M_{p^k}^{\otimes Z}\cong M_{p^\infty}^{\otimes G/H}$. This reduces the proof to the case considered above. 
\end{proof}
\begin{theorem}\label{thm-kk-equivalence}
	Let $G$ be a finite group, let $Z$ be a countable $G$-set and let $M_\mathfrak n$ be a UHF-algebra of infinite type. Then the canonical inclusions 
		\[M_{\mathfrak n}\hookrightarrow M_{\mathfrak n}\otimes M_{\mathfrak n}^{\otimes Z}\hookleftarrow M_{\mathfrak n}^{\otimes Z}\]
	are $KK^G$-equivalences, where $M_{\mathfrak n}$ is endowed with the trivial action and where $M_{\mathfrak n}^{\otimes Z}$ is endowed with the Bernoulli shift. If $Z$ is infinite, and $m<\infty$, the same conclusion holds for the inclusions
	\[
		M_{m^\infty}\hookrightarrow M_{m^\infty}\otimes M_{ m}^{\otimes Z}\hookleftarrow M_{m}^{\otimes Z}.
	\]
\end{theorem}
\begin{proof}
	 Since $M_\mathfrak n$ is strongly self-absorbing (in the sense of \cite{TW}), the map 
	\[\id_{M_\mathfrak n} \otimes 1\colon M_\mathfrak n\to M_\mathfrak n\otimes M_\mathfrak n\]
	is a $KK$-equivalence.
	Using Theorem \ref{thm-absorb}, we can identify the map
		\[\id_{M_\mathfrak n^{\otimes Z}} \otimes 1\colon M_\mathfrak n^{\otimes Z}\hookrightarrow  M_\mathfrak n^{\otimes Z}\otimes M_\mathfrak n\]
	with the map
		\[\id_{M_\mathfrak n^{\otimes Z}}\otimes (\id_{M_\mathfrak n}\otimes 1) \colon M_\mathfrak n^{\otimes Z}\otimes M_\mathfrak n\to 	M_\mathfrak n^{\otimes Z}\otimes M_\mathfrak n\otimes M_\mathfrak n,\]
	which is a $KK^G$-equivalence.
	 Similarly, if $Z$ is infinite and $m<\infty$, the map
			\[M_m^{\otimes Z}\hookrightarrow M_{m^\infty} \otimes M_m^{\otimes Z}\]
			is a $KK^G$-equivalence. 
	We prove that the map
		\begin{equation}\label{eq-infinite-stage}
			\id_{M_\mathfrak n}\otimes 1_{M_\mathfrak n^{\otimes Z}}\colon M_\mathfrak n\to M_\mathfrak n\otimes M_\mathfrak n^{\otimes Z}
		\end{equation}
	is a $KK^G$-equivalence. Note that this map is the inductive limit of the maps 
		\begin{equation}\label{eq-finite-stage}
			\id_{M_\mathfrak n}\otimes 1_{M_k^{\otimes Y}}\colon M_\mathfrak n\to M_\mathfrak n\otimes M_k^{\otimes Y}
		\end{equation}
	where $k$ ranges over all positive integers that divide $\mathfrak n$ and where $Y$ ranges over all finite $G$-subsets of $Z$. It follows from the finiteness of $G$, the nuclearity of the involved algebras and \cite[Proposition 2.6, Lemma 2.7]{MN} that the map in \eqref{eq-infinite-stage} is also the homotopy colimit (with respect to the triangulated structure of $KK^G$) of the maps in \eqref{eq-finite-stage}. Since a homotopy colimit of $KK^G$-equivalences is a $KK^G$-equivalence\footnote{This follows from the axioms of a triangulated category. The fact that homotopy colimits of maps are not unique does not cause a problem here. The reader may alternatively perform the same argument in the $\infty$-category $KK^G_{\mathrm{sep}}$ introduced in \cite{Bunke2023a}.}, it suffices to show that the maps appearing in \eqref{eq-finite-stage} are $KK^G$-equivalences. 
	
	Note that $\ell^2\left(\{1,\dotsc,k\}^Y\right)$ implements an equivariant Morita equivalence between $M_k^{\otimes Y}$ and $\C$ which maps the class of the inclusion $\C\to M_k^{\otimes Y}$ in $KK^G(\C,M_k^{\otimes Y})$ to the class $[\pi_k]\in KK^G(\C,\C)$ of the permutation representation $\pi_k\colon G\to \GL\left(\ell^2\left(\{1,\dotsc,k\}^Y\right)\right)$.
	Therefore, the maps in \eqref{eq-finite-stage} can be identified with the elements $[\id_{M_\mathfrak n}]\otimes_\C [\pi_k]\in KK^G(M_\mathfrak n,M_\mathfrak n)$.
	
	 By Proposition \ref{prop-KK-elements}, there is an element $\beta\in KK^G(\C,\C)$ and $l\geq 1$ such that $[\pi_k] \beta=k^l$. Thus $[\id_{M_\mathfrak n}]\otimes_\C [\pi_k]$ is invertible with inverse $\frac 1 {k^{l}}[\id_{M_\mathfrak n}]\otimes_\C \beta$.  The same proof shows that, if $Z$ is infinite and $m<\infty$, the map
	\[
			\id_{M_{m^\infty}}\otimes 1_{M_m^{\otimes Z}}\colon M_{m^\infty}\to M_{m^\infty} \otimes M_m^{\otimes Z}
\]
	is a $KK^G$-equivalence.
\end{proof}
\begin{remark}\label{rem-cocycle-conjugate}
	By \cite[Theorem B]{GardellaLupini} and \cite[Theorem 4.9]{matui2014ƶ}, a countable discrete group $G$ is amenable if and only if for some (any) supernatural number $\mathfrak n\neq 1$ of infinite type, the Bernoulli shift on $M_{\mathfrak n}^{\otimes G}$ absorbs the trivial action on the Jiang-Su algebra $\mathcal Z$ up to cocycle conjugacy. In particular (since $M_\mathfrak n\cong M_\mathfrak n\otimes \mathcal Z$), the conclusion of Theorem \ref{thm-absorb} is false for non-amenable groups. On the other hand, Theorem \ref{thm-kk-equivalence} together with the Higson-Kasparov Theorem \cite{HK} (applied in the form of \cite[Theorem 8.5]{MN}) implies that if $G$ is a countable amenable group, then $M_{\mathfrak n}^{\otimes G}$ absorbs the trivial action on $M_\mathfrak n$ up to $KK^G$-equivalence. It is thus conceivable that a countable discrete group $G$ is amenable if and only if the Bernoulli shift on $M_\mathfrak n^{\otimes G}$ absorbs the trivial action on $M_\mathfrak n$ up to cocycle conjugacy. 
	
	The following observation provides some evidence for this: Let $G$ be a countable amenable group, $\mathfrak n\neq 1$ a supernatural number of infinite type, and $A$ a $G$-$C^*$-algebra. By the remarks above, the unital embedding
		 \[\id\otimes 1\colon(A\otimes M_\mathfrak n^{\otimes G})\rtimes G\hookrightarrow (A\otimes M_\mathfrak n^{\otimes G})\rtimes G\otimes M_\mathfrak n\]
	is a $KK$-equivalence between $\mathcal Z$-stable $C^*$-algebras that induces an isomorphism on the trace spaces, in particular it induces an isomorphism on the Elliott invariants. If we additionally assume that $(A\otimes M_\mathfrak n^{\otimes G})\rtimes G$ is simple, separable, nuclear, and satisfies the UCT (which happens in many cases of interest), then the classification of unital, simple, separable, nuclear, $\mathcal Z$-stable $C^*$-algebras satisfying the UCT \cite{KirchbergPhillips,EGLN,TWW,CETWW,Carrion2023} implies that $(A\otimes M_\mathfrak n^{\otimes G})\rtimes G\cong (A\otimes M_{\mathfrak n}^{\otimes G})\rtimes G\otimes M_\mathfrak n$. This condition is certainly necessary for $M_\mathfrak n^{\otimes G}$ to absorb $M_\mathfrak n$ up to cocycle conjugacy.
\end{remark}
\begin{corollary}\label{cor-Rokhlin}
	Let $G\not=\{e\}$ be a finite group, let $Z$ be a $G$-set and let $M_\mathfrak n$ be a UHF-algebra of infinite type. Then the Bernoulli shift of $G$ on $M_{\mathfrak n}^{\otimes Z}$ does not have the Rokhlin property. 
\end{corollary}
\begin{proof}
	 Assume the contrary. Without loss of generality, $\fn\not=1$. 
	 Then \cite[Theorem 3.13]{Izumi2} (which is applicable by the combination of \cite[Proposition 7.1.3]{Phillips} and \cite[Theorem 3.1]{Kishimoto}) yields an isomorphism
		\[K_0\left(M_{\mathfrak n}^{\otimes Z}\rtimes G\right)\cong K_0\left(M_\mathfrak n^{\otimes Z}\right)=\Z[1/\mathfrak n].\]
	On the other hand, Theorem \ref{thm-kk-equivalence} yields an isomorphism\footnote{This $K$-theoretic statement follows from the countable case by taking inductive limits over all countable $G$-subsets of $Z$.}
		\[K_0\left(M_{\mathfrak n}^{\otimes Z}\rtimes G\right)\cong K_0(C^*(G))\otimes_{\Z} \Z[1/\mathfrak n]\cong \Z[1/\mathfrak n]^{\oplus \hat G},\]
	a contradiction.	
\end{proof}
The following corollary is particularly useful for analyzing the K-theory of Bernoulli shift actions of infinite groups and plays a crucial role in the proof of \cite[Theorem A]{CEKN}. We refer to \cite{BCH} for the formulation of the Baum--Connes conjecture with coefficients. Note that the Baum-Connes conjecture with coefficients holds for many groups, including a-T-menable groups \cite{HK} and hyperbolic groups \cite{Lafforgue}. 
\begin{corollary}\label{cor-BCC}
	Let $G$ be a countable discrete group satisfying the Baum--Connes conjecture with coefficients, let $Z$ be a $G$-set, let $A$ be a $G$-$C^*$-algebra and let $M_\mathfrak n$ be a UHF-algebra. Assume that $Z$ is infinite or that $\mathfrak n$ is of infinite type. Then the inclusion $A\to A\otimes M_\mathfrak n^{\otimes Z}$ induces an isomorphism 
		\[K_*\left(A\rtimes_r G\right)[1/\mathfrak n]\cong K_*\left(\left(A\otimes M_\mathfrak n^{\otimes Z}\right)\rtimes_r G\right).\]
	In particular, the right hand side is a $\Z[1/\mathfrak n]$-module.
\end{corollary}
\begin{proof}
By an inductive limit argument, we may assume that $Z$ is countable and $A$ is separable.
	If $G$ is finite, the statement follows from Theorem \ref{thm-kk-equivalence} considering the commutative diagram
	\begin{equation}\label{eq-diagram-KKG}
\xymatrix{
 A \ar[r] \ar[dr]_-{}&  A\otimes M_\mathfrak n^{\otimes Z}  \ar[r]^-{\phi_1} & A\otimes M_{{\mathfrak n}^\infty} \otimes M_\mathfrak n^{\otimes Z}   \\
 &  A\otimes M_{{\mathfrak n}^\infty} \ar[ur]_-{\phi_2} & 
}
	\end{equation}
	where $\phi_1, \phi_2$ are $KK^G$-equivalences.
	 Assume now that $G$ is infinite. Consider the diagram \eqref{eq-diagram-KKG}. We know that (the restrictions of) $\phi_1, \phi_2$ are $KK^H$-equivalences for every finite subgroup $H\subset G$. Since $G$ satisfies the Baum--Connes conjecture with coefficients, the results of \cite{CEO} (c.f. \cite{MN}) imply that $\phi_1$ and $\phi_2$ induce isomorphisms of the $K$-theory groups of reduced crossed products by $G$. The statement follows from this by identifying $K_\ast((A\otimes M_{{\mathfrak n}^\infty}) \rtimes_rG) \cong K_\ast(A \rtimes_rG)[1/ \mathfrak n]$.
\end{proof}
We end this section with an application to Bernoulli shifts on strongly self-absorbing $C^*$-algebras. Recall that a separable, unital $C^*$-algebra $\mathcal{D}\neq \C$ is strongly self-absorbing \cite{TW} if there is an isomorphism $\mathcal{D}\cong \mathcal{D}\otimes \mathcal{D}$ which is approximately unitarily equivalent to the first factor inclusion $\mathrm{id}_{\mathcal{D}}\otimes 1_{\mathcal{D}}\colon \mathcal{D}\to \mathcal{D}\otimes \mathcal{D}$. Strongly self-absorbing $C^*$-algebras are automatically simple, nuclear \cite{TW} and $\mathcal Z$-stable \cite{Winter}.
By the combination of \cite[Proposition 5.1]{TW} and the classification of unital, simple, separable, nuclear, $\mathcal Z$-stable $C^*$-algebras in the UCT class \cite{KirchbergPhillips,EGLN,TWW,CETWW,Carrion2023}, a complete list of strongly self-absorbing $C^*$-algebras satisfying the UCT is given by
\begin{equation} \label{eq-list}
\mathcal{Z}, \, M_\mathfrak n,\,  \mathcal{O}_\infty,\, \mathcal{O}_\infty\otimes M_\mathfrak n,\,  \mathcal{O}_2,
\end{equation}
where $\mathfrak n\neq1$ is a supernatural number of infinite type. 

The following corollary is a generalization of \cite[Corollary 6.9]{Szabo}.
We refer to \cite[Section 8]{MN} for a definition of \emph{having a $\gamma$-element equal to $1$} (where $X=\mathrm{pt}$ in our case). An equivalent way of phrasing this definition is that any element $f\in KK^G(A,B)$ which is a $KK^H$-equivalence for all finite subgroups $H\subset G$ is a $KK^G$-equivalence \cite[Theorem 8.3]{MN}.
By the Higson--Kasparov theorem \cite{HK} this assumption is satisfied for all a-T-menable groups. 
\begin{corollary} \label{cor-ssa}
Let $\mathcal{D}$ be a strongly self-absorbing $C^*$-algebra satisfying the UCT and let $G$ be a countable discrete group having a $\gamma$-element equal to $1$. Then, for any countable $G$-set $Z$, the $G$-$C^*$-algebra $\mathcal{D}^{\otimes Z}$ equipped with the Bernoulli shift is $KK^G$-equivalent to $\mathcal{D}$ equipped with the trivial $G$-action.
\end{corollary}
For the proof, we need the following result of Izumi \cite{Izumi} which we spell out here for later reference. 
\begin{theorem}[{\cite[Theorem 2.1]{Izumi}}, see also {\cite[Lemma 6.8]{Szabo}}]\label{thm-Izumi}
	Let $A,B$ be separable nuclear $C^*$-algebras, let $H$ be a finite group and let $Z$ be a finite $H$-set. Then, there is a map from $KK(A, B)$ to  $KK^H(A^{\otimes Z}, B^{\otimes Z})$ which in particular, sends the class of a $\ast$-homomorphism $\phi$ to the class of $\phi^{\otimes Z}$. Furthermore, this map is compatible with the compositions and in particular sends a $KK$-equivalence to a $KK^H$-equivalence. In particular, the Bernoulli shifts on $A^{\otimes Z}$ and $B^{\otimes Z}$ are $KK^H$-equivalent if $A$ and $B$ are $KK$-equivalent. 
\end{theorem}
\begin{remark}\label{rem_Izumi} See \cite[Lemma 2.4]{CEKN}) or \cite[Theorem 2.17]{Bunke2023} for a generalization of Theorem \ref{thm-Izumi}. In particular, the nuclearity assumption is not necessary. 
\end{remark}
\begin{proof}[Proof of Corollary \ref{cor-ssa}]
	We claim that the unital embeddings
		\begin{equation}\label{eq-D-inclusions}
			\mathcal D\hookrightarrow \mathcal D\otimes \mathcal D^{\otimes Z}\hookleftarrow \mathcal D^{\otimes Z}
		\end{equation}
	are $KK^G$-equivalences. By the assumption on $G$, this amounts to showing that they are $KK^H$-equivalences for every finite subgroup $H\subset G$. By the same homotopy co-limit argument as in the proof of Theorem \ref{thm-kk-equivalence}, it is enough to show that the maps 
	\[
			\mathcal D\hookrightarrow \mathcal D\otimes \mathcal D^{\otimes Y}\hookleftarrow \mathcal D^{\otimes Y}
	\]
	are $KK^H$-equivalences for all finite $H$-subsets $Y$ of $Z$. 
	Now Theorem \ref{thm-Izumi} allows us to replace $\mathcal D$ by a $KK$-equivalent $C^*$-algebra. Thanks to the list \eqref{eq-list}, this reduces the problem to the cases $\mathcal D=\C$, $\mathcal D=0$ and $\mathcal D=M_\mathfrak n$. The first two cases are trivial and the third one follows from Theorem \ref{thm-kk-equivalence}.
\end{proof}

\section{Equivariantly $KK$-trivial flips}\label{sec-trivial} 
Recall that a $C^*$-algebra $A$ is said to have approximately inner flip if the flip automorphism $ A\otimes A \to A\otimes A, \,\,\, a\otimes b\mapsto b\otimes a$ is approximately inner, i.e. a point-norm limit of inner automorphisms. A $C^*$-algebra $A$ with approximately inner flip must be simple, nuclear and have at most one trace \cite{ER}.  An approximately inner flip necessarily induces the identity map on $K_\ast(A\otimes A)$ and this largely restricts the class of $C^*$-algebras $A$ with approximately inner flip. Effros and Rosenberg \cite{ER} showed that if $A$ is AF, then $A$ must be stably isomorphic to a UHF-algebra. Tikuisis \cite{Tikuisis} determined a complete list of classifiable $C^*$-algebras with approximately inner flip. We would like to thank Dominic Enders, Andr\'e Schemaitat and Aaron Tikuisis for informing us about a corrigendum stated below:
\begin{theorem} \label{thm-aif}(\cite[Theorem 1.3]{TikuisisCorrigendum}, Correction to \cite[Theorem 2.2]{Tikuisis})  Let $A$ be a separable, unital $C^*$-algebra with strict comparison, in the UCT class, which is either infinite or quasidiagonal. The following are equivalent:
\begin{enumerate}
\item\label{item-aif} $A$ has approximately inner flip;
\item\label{item-list} $A$ is Morita equivalent to one of the following $C^*$-algebras:
\begin{itemize}
\item $\mathbb{C}$;
\item $\mathcal{E}_{\mathfrak n, 1, \mathfrak m}$;
\item $\mathcal{E}_{\mathfrak n, 1, \mathfrak m}\otimes \mathcal{O}_\infty$;
\item $\mathcal F_{1, \mathfrak m}$.
\end{itemize}
\end{enumerate}
Here $\mathfrak m, \mathfrak n$ are supernatural numbers with $\mathfrak{m}$ of infinite type such that $\mathfrak m$ divides $\mathfrak n$, $\mathcal{O}_\infty$ is the Cuntz algebra on infinitely many generators, $\mathcal{E}_{\mathfrak n, 1,\mathfrak m}$ is the simple, separable, unital, $\mathcal{Z}$-stable, nuclear\footnote{In \cite{Tikuisis}, quasidiagonality was assumed as well but this is redundant by \cite{TWW}.} $C^*$-algebra in the UCT class with unique trace satisfying
\[
K_0(\mathcal{E}_{\mathfrak n, 1, \mathfrak m}) \cong \Q_\mathfrak n,\,\,  [1]_0=1, \,\, K_1(\mathcal{E}_{\mathfrak n, 1,\mathfrak  m}) \cong \Q_\mathfrak m/\Z,
\]
and $\mathcal F_{1,\mathfrak m}$ is the unique unital Kirchberg algebra in the UCT class satisfying 
\[
K_0(\mathcal F_{1,\mathfrak m})\cong 0,\,\,\, K_1(\mathcal F_{1,\mathfrak m})\cong \Q_\mathfrak m/\Z.\]
\end{theorem}
The following is part of what is proven in \cite{TikuisisCorrigendum}, \cite{Tikuisis}:
\begin{theorem}[\cite{TikuisisCorrigendum}, \cite{Tikuisis}] \label{thm_aif} Let $A$ be a separable $C^*$-algebra satisfying the UCT. Then, the following are equivalent:
\begin{enumerate}
\item\label{item-aif-1} The flip map $\sigma_{A, A}$ is equal to the identity in $KK(A\otimes A,A\otimes A)$;
\item\label{item-aif-2} As a graded abelian group, $K_0(A)\oplus K_1(A)$ is isomorphic to either $0\oplus \Q_\fm/\Z$ or $\Q_\fn\oplus \Q_\fm/\Z$ for supernatural numbers $\fm, \fn$ such that $\fm$ is of infinite type, and $\fm$ divides $\fn$; 
\item\label{item-aif-3} $A$ is $KK$-equivalent to either $\mathcal F_{1,\fm}$ or $M_{\fn} \oplus \mathcal{F}_{1, \fm}$ for supernatural numbers $\fm, \fn$ such that $\fm$ is of infinite type, and $\fm$ divides $\fn$; 
\item\label{item-aif-4} $A$ is $KK$-equivalent to a unital, simple, separable, nuclear $\mathcal Z$-stable $C^*$-algebra with approximately inner flip.
\end{enumerate}
\end{theorem}
\begin{proof}
	\eqref{item-aif-2} $\implies$ \eqref{item-aif-1}: This is \cite[Theorem 1.5]{TikuisisCorrigendum}. 
	
	\eqref{item-aif-1} $\implies$ \eqref{item-aif-2}: This follows from the proof of \cite[Theorem 1.7]{TikuisisCorrigendum} which formally assumes that $A$ has approximately inner flip, but only uses that $\sigma_{A,A}$ is equal to the identity in $KK(A\otimes A,A \otimes A)$. 
	
	\eqref{item-aif-2} $\Leftrightarrow$ \eqref{item-aif-3}: This follows from the UCT. 
	
	\eqref{item-aif-2} $\Leftrightarrow$ \eqref{item-aif-4}:  This follows from the UCT and Theorem \ref{thm-aif}. 
\end{proof}
We identify the subclass of $C^*$-algebras $A$ considered in Theorem \ref{thm_aif} for which the flip $C_2$-action on $A\otimes A$ is $C_2$-equivariantly $KK$-trivial. These are precisely the ($K$-theoretically) self-absorbing ones:
\begin{theorem} \label{thm_GKK} Let $A$ be a separable $C^*$-algebra satisfying the UCT. The following are equivalent:
	\begin{enumerate}
		\item\label{item-classif} $A$ is $KK$-equivalent to a unital, simple, separable, nuclear, $\mathcal Z$-stable $C^*$-algebra $A$ with approximately inner flip such that $A\otimes A\cong A$;
		\item\label{item-K-absorbing} The flip map $\sigma_{A,A}$ is equal to the identity in $KK(A\otimes A,A\otimes A)$ and we have an isomorphism $K_*(A)\cong K_*(A\otimes A)$; 
		\item\label{item-C2-trivial} The flip action on $A^{\otimes C_2}$ is $KK^{C_2}$-equivalent to the trivial action on $A\otimes A$;
		\item\label{item-G-trivial} For any finite group $G$, and for any finite $G$-set $Z$, $A^{\otimes Z}$ equipped with the Bernoulli shift $G$-action is $KK^G$-equivalent to $A$ equipped with the trivial $G$-action;
		\item\label{item-K-theory} As a graded abelian group, $K_0(A)\oplus K_1(A)$ is isomorphic to either $0\oplus \Q_\fm/\Z$ or $\Q_\fn\oplus \Q_\fm/\Z$ for supernatural numbers $\fm, \fn$ of infinite type such that $\fm$ divides $\fn$. 
	\end{enumerate}
	In particular, for any $G$ and $Z$ as in \eqref{item-G-trivial}, we have an isomorphism of $R_\C(G)$-modules
	\[K_*(A^{\otimes Z}\rtimes_r G)\cong K_*(A)\otimes_\Z R_\C(G),\]
	where $R_\C(G)\coloneqq K_0(C^*(G))\cong \Z[\hat G]$ denotes the complex representation ring. 
\end{theorem}
For the proof, we use the following definition:
\begin{definition}\label{def-ess}
	A supernatural number $\mathfrak n$ is said to be of \emph{essentially infinite type} if it can be decomposed as $\mathfrak{n} = n_0 \cdot \mathfrak{n}_1$, where $n_0$ is a natural number and $\mathfrak{n}_1$ is a supernatural number of infinite type.
\end{definition}
\begin{remark}\label{rem-ess}
Note that for $\mathfrak n=n_0 \cdot \mathfrak n_1$ as above, we have $\Q_\mathfrak n\cong \Q_{\mathfrak n_1}$ as abelian groups. 
\end{remark}
\begin{lemma}\label{lem-mult}
	Let $\fm,\fn$ be supernatural numbers. 
	Then the multiplication map $\Q_\mathfrak n \otimes_\Z \Q_\mathfrak m\to \Q_{\fn\fm}$ is an isomorphism. 
	Moreover, $\Q_{\fn}\cong \Q_{\fn^2}$ if and only if $\fn$ is of essentially infinite type. 
\end{lemma}
\begin{proof}
	The first statement follows by writing both sides as appropriate inductive limits of the multiplication map $\Z\otimes_\Z\Z\to \Z$. 
	For the second statement let $\mathfrak n=\prod_{i=1}^\infty p_i^{n_{p_i}}$ be a supernatural number such that $1\leq n_{p_i}<\infty$ for infinitely many distinct primes $p_i$ and write $q_i=p_i^{n_{p_i}}$. 
	Then any element in $\Q_\mathfrak {n^2}$ can be divided by $q_i^2$ for infinitely many $i$, which is not the case for  $\Q_\mathfrak {n}$. For the converse, assume that $\mathfrak n= n_0 \cdot \mathfrak n_1$ where $n_0$ is a natural number and $\mathfrak n_1$ is of infinite type. Then we have $\Q_{n_0}\cong \Z \cong \Q_{n_0^2}$ and $\fn_1=\fn_1^2$ and thus 
	\[\Q_{\fn} \cong  \Q_{n_0}\otimes_\Z \Q_{\fn_1} \cong \Q_{n_0^2}\otimes_{\Z}\Q_{\fn_1^2}\cong \Q_{\fn^2}.\]
\end{proof}
\begin{lemma}\label{lem-self-absorbing}
	Let $\fm,\fn$ be supernatural numbers such that $\fm$ is of infinite type and $\fm$ divides $\fn$. 
	\begin{enumerate}
		\item\label{item-selfabs1} We have $\mathcal F_{1,\fm}\cong \mathcal F_{1,\fm}\otimes \mathcal F_{1,\fm}$;
		\item\label{item-selfabs2} We have $K_*(\mathcal E_{\fn,1,\fm}\otimes \mathcal E_{\fn,1,\fm}) \cong \begin{cases}
			\Q_{\fn^2},&*=0\\
			\Q_{\fm}/\Z,&*=1
		\end{cases}$ with $[1_{\mathcal E_{\fn,1,\fm}\otimes \mathcal E_{\fn,1,\fm}}]_0=1\in \Q_{\fn^2}$;
		\item\label{item-selfabs3} We have $\mathcal E_{\fn,1,\fm}\cong \mathcal E_{\fn,1,\fm}\otimes \mathcal E_{\fn,1,\fm}$ if and only if $\mathfrak n$ is of infinite type.  
	\end{enumerate}
\end{lemma}
\begin{proof}
	This is already remarked in the paragraph after \cite[Proposition 7.3]{Tikuisis}. For the convenience of the reader, we give the proof below.
	
	\eqref{item-selfabs1} follows from the combination of the K\"unneth theorem \cite{Rosenberg1987} and the Kirchberg--Phillips classification theorem \cite{KirchbergPhillips}. 
	
	\eqref{item-selfabs2} follows from the K\"unneth theorem and Lemma \ref{lem-mult}.
	
	\eqref{item-selfabs3}: It easily follows from \eqref{item-selfabs2} that there is a unit-preserving isomorphism  
	${K_0(\mathcal E_{\fn,1,\fm}\otimes \mathcal E_{\fn,1,\fm})\cong K_0(\mathcal E_{\fn,1,\fm})}$
	if and only if $\fn$ is of infinite type. Indeed, the only unit preserving homomorphism $\mathbb{Q}_\fn \to \mathbb{Q}_{\fn^2}$ is the canonical inclusion $\mathbb{Q}_\fn \subset \mathbb{Q}_{\fn^2}$, and it is an isomorphism if and only if $\fn$ is of infinite type.
	Moreover, since $\mathcal E_{\fn,1,\fm}$ has a unique trace, the classification theorem \cite{EGLN,TWW,CETWW,Carrion2023} implies that such a unit-preserving isomorphism is necessarily induced by an isomorphism of the underlying $C^*$-algebras. 
\end{proof}
\begin{lemma}\label{lem_Kcompute} Let $A=M_{\fn}$ for a supernatural number $\mathfrak n$. Suppose that
\[
K_\ast\left(A^{\otimes C_2}\rtimes C_2\right) \cong K_\ast(A^{\otimes 2}\otimes C^*_r(C_2)).
\]
Then, $\mathfrak n$ is of essentially infinite type.
\end{lemma}
\begin{proof}
We prove the contrapositive. 
Suppose first that $\mathfrak n= \prod_{i=1}^{\infty} p_i^{n_{p_i}}$ where $1\leq n_{p_i}<\infty$ for infinitely many primes $p_i$ and write $q_i=p_i^{n_{p_i}}$. 
We first assume that $n_{p_i}<\infty$ for \emph{all} $i\in \N$. 
Then, $M_\mathfrak n^{\otimes C_2}\rtimes C_2$ is the inductive limit of the system
\[
\mathbb{C}\rtimes C_2 \xrightarrow{}  M_{q_1}^{\otimes C_2}\rtimes C_2  \xrightarrow{} (M_{q_1}\otimes M_{q_2})^{\otimes C_2}\rtimes C_2   \xrightarrow{}\dotsb
\]
From this (c.f. Proof of Theorem~\ref{thm-kk-equivalence}), we observe that  $K_0(M_\mathfrak n^{\otimes C_2}\rtimes C_2)$ is isomorphic to the inductive limit of the system
\[
R_{\mathbb{C}}(C_2)\xrightarrow{\cdot [\pi_{q_1}]} R_{\mathbb{C}}(C_2) \xrightarrow{\cdot [\pi_{q_2}]} R_{\mathbb{C}}(C_2) \xrightarrow{\cdot [\pi_{q_3}]}\dotsb
\]
where $\pi_k\colon C_2 \to  \GL\left(\ell^2\left(\{1,\dotsc,k\}^{C_2}\right)\right)$ is the permutation representation. We identify $R_{\mathbb{C}}(C_2) \cong  \mathbb{Z}^2$ using the trivial representation $[\sigma_0]$ and the sign representation $[\sigma_1]$ of $C_2$ as a basis of $R_{\mathbb{C}}(C_2)$. Since $[\pi_k] =  \frac{k(k+1)}{2} [\sigma_0] +  \frac{k(k-1)}{2} [\sigma_1]$ in $R_{\mathbb{C}}(C_2)$ (by the same arguments as in the proof of Proposition \ref{prop-KK-elements}), we see that the system is isomorphic to
\[
\mathbb{Z}^2 \xrightarrow{X_{q_1} \cdot } \mathbb{Z}^2  \xrightarrow{X_{q_2}  \cdot } \mathbb{Z}^2  \xrightarrow{X_{q_3}  \cdot } \dotsb
\]
where $X_k = \begin{bmatrix} \frac{k(k+1)}{2} & \frac{k(k-1)}{2} \\  \frac{k(k-1)}{2} & \frac{k(k+1)}{2} \end{bmatrix}$, which has eigenvectors $\begin{bmatrix} 1 \\ 1 \end{bmatrix}$,  $\begin{bmatrix} 1 \\ -1 \end{bmatrix}$ and the corresponding eigenvalues $k^2, k$. The system has a subsystem consiting of the span of $\begin{bmatrix} 1 \\ 1 \end{bmatrix}$ in each $\mathbb{Z}^2$ on which $X_{q_k}$ acts as $q_k^2$. The quotient system is isomorphic to
\[
\mathbb{Z}  \xrightarrow{q_1 \cdot } \mathbb{Z}  \xrightarrow{q_2  \cdot } \mathbb{Z}  \xrightarrow{q_3  \cdot } \dotsb.
\]
These induce the following short exact sequence 
\[
\xymatrix{
	0 \ar[r] & \Q_\mathfrak {n^2}  \ \ar[r] &  K_0(M_\mathfrak n^{\otimes C_2}\rtimes C_2) \ar[r] &   \Q_\mathfrak {n} \ar[r] &  0.
}
\]
By our assumption on $\fn$, the same reasoning as in the proof of Lemma \ref{lem-mult} shows that 
\[
K_0(M_\mathfrak n^{\otimes C_2}\rtimes C_2)\ncong   \Q_\mathfrak {n^2}\oplus  \Q_\mathfrak {n^2} \cong K_0(M_\mathfrak n^{\otimes 2}\otimes C^*_r(C_2)).
\]
These conclusions remain to hold for $\mathfrak{n}= \mathfrak{n}_0\cdot \mathfrak{n}_1$ where $\mathfrak{n}_0$ is a supernatural number of the type we just considered and $\mathfrak{n}_1$ is a supernatural number of infinite type which is coprime to $\mathfrak{n}_0$. This can be seen by either generalizing the argument above or by applying Theorem \ref{thm-kk-equivalence} and noting that $\Q_{\mathfrak n^2}\cong \Q_{\mathfrak n}$ implies $\Q_{\mathfrak n_0^2}\cong \Q_{\mathfrak n_0}$. 
\end{proof}

The following ideal filtration is a key technical ingredient for our proof of Theorem \ref{thm_GKK}. 
Special cases of this technique can be found implicitly in \cite{Izumi,CEKN,Bunke2023}. 
Since this is a recurring theme and might be of independent interest, we formlalize it as a general proposition here.
\begin{proposition}\label{prop_Izumi_filtration} Let $G$ be a locally compact group and, let $Z$ be a finite $G$-set. Let
\begin{equation}\label{eq_Izumi_seq}
\xymatrix{
0 \ar[r] & J \ar[r]^{j} &  B  \ar[r]^{\pi} & B/J \ar[r] &  0
}
\end{equation}
be a short exact sequence of $C^*$-algebras which admits a contractive, completely positive (c.c.p.) splitting. Define $I_0\coloneqq B^{\otimes Z}$ equipped with the Benoulli shift $G$-action. Let 
\[
0=I_{|Z|+1} \subset I_{|Z|} \subset \ldots  \subset I_{j+1} \subset I_{j} \subset \ldots \subset I_0
\]
be a $G$-equivariant filtration of $I_0$, where $I_j$ is the $G$-invariant ideal of $B^{\otimes Z}$ generated by elementary tensors
\[
\otimes_{z\in Z} b_z
\]
where at least $j$-many of $b_z$ belong to $J$. In particular, we have $I_{|Z|}=J^{\otimes Z}$. By convention, $I_{|Z|+1}\coloneqq 0$. Then, for any $0\leq j \leq |Z|$, we have a canonical $G$-equivariant isomorphism
\begin{equation}\label{eq-Izumi-graded}
I_{j}/I_{j+1} \cong \bigoplus_{F \subset Z, |F|=j} J^{\otimes F} \otimes (B/J)^{\otimes Z-F},
\end{equation}
where the right-hand side is endowed with the canonically induced $G$-aciton: $g\in G$ maps $J^{\otimes F} \otimes (B/J)^{\otimes Z-F}$ to $J^{\otimes g(F)} \otimes (B/J)^{\otimes Z-g(F)}$ by permuting the tensor factors. Moreover, the sequence
\[
0 \to I_{j+1} \to I_j \to I_{j}/I_{j+1} \to 0
\]
admits a $G$-equivariant c.p. splitting. If $B$ is separable, the sequence admits a $G$-equivariant c.c.p. splitting. 

\end{proposition}
\begin{proof}  Note that all the involved actions of $G$ factor through the symmetric group on $Z$, which is finite. 
Hence, we will assume that $G$ is finite throughout (this assumption will be relevant only in the last part). 

We first prove \eqref{eq-Izumi-graded} for $j=0$ by establishing that $\ker(\pi^{\otimes Z})=I_1$.
The inclusion $\ker(\pi^{\otimes Z})\supset I_1$ is trivial. 
For the reverse inclusion, let $s\colon B/J \to B$ be a c.c.p. map that splits $\pi\colon B\to B/J$. 
Since $\left(\prod_{i \in F}(1-x_i)\right)_{F\subset Z}$ is a basis for the subspace of affine multilinear polynomials in $\bZ[x_i \mid i\in Z]$, we have
\begin{equation}\label{eq-polynomialproduct}
1- \prod_{i \in Z}x_i = \sum_{F \subset Z} \left( a_F \prod_{i \in F}(1-x_i) \right)
\end{equation}
for unique $a_F\in \bZ$ for $F\subset Z$ and $a_\emptyset=0$ (to see this, substitute $x_i=1$ in \eqref{eq-polynomialproduct}). It follows that we have
\begin{equation*}\label{eq-1minusZ}
\mathrm{id}_{B^{\otimes Z}} - (s\circ \pi)^{\otimes Z}  = \sum_{\emptyset \neq F \subset Z} a_F (\mathrm{id}_B -s \circ \pi)^{\otimes F} \otimes  \mathrm{id}_{B^{\otimes Z-F}}   
\end{equation*}
on $B^{\otimes Z}$. 
From this, it is easy to see that $\mathrm{im}(\mathrm{id}_{B^{\otimes Z}} - (s\circ \pi)^{\otimes Z})\subset I_1$. 
Moreover, since the $G$-equivariant c.c.p. map $s^{\otimes Z}\colon (B/J)^{\otimes Z} \to B^{\otimes Z}$ splits $\pi^{\otimes Z}\colon B^{\otimes Z} \to (B/J)^{\otimes Z}$, we have $\ker(\pi^{\otimes Z})=\mathrm{im}(\mathrm{id}_{B^{\otimes Z}} - (s\circ \pi)^{\otimes Z})$.

Thus, we have a canonical $G$-equivariant isomorphism
\[
I_0/I_1 = B^{\otimes Z}/\ker(\pi^{\otimes Z})\cong (B/J)^{\otimes Z}.
\]
To prove \eqref{eq-Izumi-graded} for $0 < j < |Z|$, note that the subalgebras $J^{\otimes F_i}\otimes B^{\otimes Z-F_i} \subset I_j$ for distinct $F_i\subset Z$ with $|F_i|=j$ are pairwise orthogonal modulo $I_{j+1}$. 
Since these subalgebras generate $I_j$, it follows that the quotient $I_j/I_{j+1}$ is the direct sum of the quotient of $J^{\otimes F}\otimes B^{\otimes Z-F}$ by $(J^{\otimes F}\otimes B^{\otimes Z-F})\cap I_{j+1}$ for $F\subset Z$ with $|F|=j$. 
The proof of \eqref{eq-Izumi-graded} for $0 < j < |Z|$ will thus follow from the case $j=1$ proved above once we show that 
\begin{equation}\label{eq-prod-ideal}
\left( J^{\otimes F}\otimes B^{\otimes Z-F} \right) \cap I_{j+1} = J^{\otimes F}  \otimes  I_{1, Z-F},
 \end{equation}
where $I_{1, Z-F}\subset B^{\otimes Z-F}$ is the ideal generated by elementary tensors $\otimes_{z\in Z-F}b_z$ with $b_z\in J$ for at least one $z\in Z-F$. 
The inclusion $\supset$ in \eqref{eq-prod-ideal} is obvious. 
The reverse inclusion follows from
 \[
 \left( J^{\otimes F}\otimes B^{\otimes Z-F} \right) \cap I_{j+1}  =  \left( J^{\otimes F}\otimes B^{\otimes Z-F} \right) \cdot  I_{j+1} \subset  J^{\otimes F}  \otimes  I_{1, Z-F},
 \]
which can be verified on the generators of $I_{j+1}$. 
This proves \eqref{eq-prod-ideal} and implies that we have a canonical $G$-equivariant isomorphism
 \[
 I_{j}/I_{j+1} \cong \bigoplus_{F \subset Z, |F|=j} J^{\otimes F} \otimes (B/J)^{\otimes Z-F}.
 \]
This finishes the proof of \eqref{eq-Izumi-graded} since the case $j=|Z|$ holds by definition. 
 
 We show that the quotient map $I_{j} \to I_{j}/I_{j+1}$ admits a $G$-equivariant c.p.c. splitting if $B$ is separable and a $G$-equivariant c.p. splitting in general.
 First, note that the sum of c.c.p. maps
 \[
 \mathrm{id}_{J^ {\otimes F}}\otimes s^{\otimes Z-F} \colon J^{\otimes F} \otimes (B/J)^{\otimes Z-F} \to  J^{\otimes F} \otimes B^{\otimes Z-F}  \to I_j,
 \]
 over $F\subset Z$ with $|F|=j$, is a $G$-equivariant c.p. splitting of  $I_j  \to  I_{j}/I_{j+1}$. 
 Now suppose $B$ is separable. Then we recall that any c.p. splitting can be modified to a (not necessarily $G$-equivariant) c.c.p. splitting (see \cite[Remark 2.5]{CS86} and also \cite{Arveson}) for separable $C^*$-algebras.
 Finally, since $G$ was assumed to be finite without loss of generality, by averaging over $G$, any not necessarily $G$-equivariant c.c.p. splitting can be promoted to a $G$-equivariant c.c.p. splitting.
\end{proof}

\begin{definition}\label{def_Izumi_filtration} We call the $G$-equivariant filtration of $B^{\otimes Z}$ by the ideals $I_j$ in Proposition \ref{prop_Izumi_filtration}, the \emph{Izumi filtration} of $B^{\otimes Z}$ associated with the short exact sequence \eqref{eq_Izumi_seq}.
\end{definition}

\begin{proof}[Proof of Theorem \ref{thm_GKK}] 
\eqref{item-K-absorbing} $\implies$ \eqref{item-classif}: 
By Theorems \ref{thm-aif} and \ref{thm_aif}, and Lemmas \ref{lem-mult} and \ref{lem-self-absorbing}, $A$ is $KK$-e\-qui\-va\-lent to either $\mathcal F_{1,\fm}$ or $\mathcal E_{\mathfrak n,1,\mathfrak m}$ for supernatural numbers $\fm,\fn$ where $\fm$ is of infinite type, $\fn$ is of essentially infinite type and $\fm$ divides $\fn$. By Remark \ref{rem-ess} we can assume that $\fn$ is of infinite type. Now the claim follows from Lemma \ref{lem-self-absorbing}.

\eqref{item-classif} $\implies$ \eqref{item-K-theory}: 
By Theorem \ref{thm-aif}, $K_0(A)\oplus K_1(A)$ is isomorphic to either ${0\oplus\Q_\fm/\Z}$ or ${\Q_\fn\oplus \Q_\fm/\Z}$ for supernatural numbers $\fm,\fn$ such that $\fm$ is of infinite type and $\fm$ divides $\fn$.
In the latter case, the K\"unneth theorem implies $\Q_\fn \cong K_0(A)\cong K_0(A\otimes A)\cong \Q_{\fn^2}$, so that $\fn$ must be of essentially infinite type by Lemma \ref{lem-mult}. In view of Remark \ref{rem-ess}, we can take $\fn$ to be of infinite type.

\eqref{item-G-trivial} $\implies$ \eqref{item-C2-trivial}: This follows by considering $C_2$-sets $Z=C_2$ with the free $C_2$-action and $Z=\{0, 1\}$ with the trivial $C_2$-action.

\eqref{item-K-theory} $\implies$ \eqref{item-G-trivial}: 
By the UCT, $A$ is $KK$-equivalent to either ${M_\fn\oplus \mathcal F_{1,\fm}}$ or $\mathcal F_{1,\fm}$ for supernatural numbers $\fm,\fn$ of infinite type such that $\fm$ divides $\fn$. We first consider the case when $A$ is $KK$-equivalent to $\mathcal{F}_{1, \fm}$. Let
\begin{equation}\label{eq_JBseq}
\xymatrix{
0 \ar[r] & J \ar[r]^{j} &  B  \ar[r]^{\pi} & B/J \ar[r] &  0
}
\end{equation}
be a short exact sequence of separable $C^*$-algebras satisfying the UCT such that the quotient map $\pi\colon B \to B/J$ is $KK$-equivalent to the unital inclusion $\iota\colon \mathbb{C} \to M_\fm$. Assume that the sequence admits a c.c.p. splitting. For example, we can take $B$ to be the mapping cylinder of $\iota$, $\pi\colon B\to B/J=M_\fm$ to be the canonical quotient map, and $J$ to be the kernel of $\pi$. Then, it follows from the six-term exact sequence that $K_0(J)\cong0$ and $K_1(J)\cong \bQ_\fm /\bZ$. By the UCT, $J$ is $KK$-equivalent to $A$. By Theorem \ref{thm-Izumi} (see also Remark \ref{rem_Izumi}), it suffices to show that $J^{\otimes Z}$ is $KK^G$-equivalent to $J$ equipped with the trivial action.

Define $I_0\coloneqq B^{\otimes Z}$ equipped with the Benoulli shift $G$-action. Let
\[
0=I_{|Z|+1} \subset I_{|Z|} \subset \ldots  \subset I_{j+1} \subset I_{j} \subset \ldots \subset I_0
\]
be the Izumi filtration (see Definition \ref{def_Izumi_filtration}) of $B^{\otimes Z}$ associated with \eqref{eq_JBseq}. In particular, we have $I_{|Z|}=J^{\otimes Z}$. By Proposition \ref{prop_Izumi_filtration}, for any $0\leq j \leq |Z|$, we have a canonical isomorphism
\[
I_{j}/I_{j+1} \cong \bigoplus_{F \subset Z, |F|=j} J^{\otimes F} \otimes (B/J)^{\otimes Z-F}.
\]
In particular, we have $I_0/I_1 \cong (B/J)^{\otimes Z}$. By Proposition \ref{prop_Izumi_filtration}, the sequences $0\to I_{j+1}\to I_j \to I_j/I_{j+1} \to 0$ admit G-equivariant c.c.p. splittings. Hence, these are all admissible extensions in $KK^G$ (see \cite[Section 2.3]{MN}) and induce triangles (a.k.a. fiber sequences) in $KK^G$. 

We claim that for any $0\leq j \leq |Z|-1$, the natural inclusion map
\begin{equation}\label{eq_filt_KKiso}
J \otimes I_{j+1} \to J\otimes I_j
\end{equation}
is a $KK^G$-equivalence, or equivalently, that $J \otimes (I_{j}/I_{j+1})$ is $KK^G$-equivalent to zero. Here, $J$ is endowed with the trivial $G$-action. To see this, we first note that $(I_{j}/I_{j+1})$ is the direct sum of the induced algebras of the form
\[
\mathrm{Ind}_{G_F}^G (J^{\otimes F} \otimes (B/J)^{\otimes Z-F})
\]
for $F\subset Z$ with $|F|=j$ where $G_F\subset G$ is the stablizer of $F$ (the elements that fix $F$ as a subset, not necessarily pointwise). Secondly, since $B/J$ is $KK$-equivalent to $M_\fm$, $(B/J)^{\otimes Z-F}$ is $KK^{G_F}$-equivalent to $M_\fm$ equipped with the trivial action by Theorem \ref{thm-kk-equivalence}. Since $J\otimes M_\fm$ is $KK$-equivalent to zero, it follows $J \otimes (J^{\otimes F} \otimes (B/J)^{\otimes Z-F})$ is $KK^{G_F}$-equivalent to zero. By Frobenius reciprocity, we get a $KK^G$-equivalence  
\[J\otimes \mathrm{Ind}_{G_F}^G (J^{\otimes F} \otimes (B/J)^{\otimes Z-F})\simeq_{KK^G} \mathrm{Ind}_{G_F}^G(J\otimes  J^{\otimes F} \otimes (B/J)^{\otimes Z-F})\simeq_{KK^G}0.\]  

We have shown that the maps \eqref{eq_filt_KKiso} are $KK^G$-equivalences for $0\leq j \leq |Z|-1$. Therefore, their composition 
\[
J \otimes I_{|Z|} \to J\otimes I_0 
\]
is a $KK^G$-equivalence. Since $I_0=B^{\otimes Z}$ is $KK^G$-equivalent to $\mathbb{C}^{\otimes Z} \cong \mathbb{C}$, this gives a $KK^G$-equivalence from $J \otimes J^{\otimes Z}$ to $J$. 

We now show that $J \otimes J^{\otimes Z}$ is $KK^G$-equivalent to $J^{\otimes Z}$. We take the tensor product of the sequence \eqref{eq_JBseq} with $J^{\otimes Z}$:
\begin{equation}\label{eq_JBseqZ}
\xymatrix{
0 \ar[r] & J\otimes J^{\otimes Z} \ar[r]^{j \otimes \mathrm{id}} &  B \otimes J^{\otimes Z}  \ar[r]^{\pi \otimes \mathrm{id}} & B/J \otimes J^{\otimes Z} \ar[r] &  0.
}
\end{equation}
By Theorem \ref{thm-kk-equivalence}, $(B/J)^{\otimes Z}$ is $KK^G$-equivalent to $B/J$. We thus have $KK^G$-equivalences
\[
 B/J \otimes J^{\otimes Z} \simeq_{KK^G}  (B/J)^{\otimes Z} \otimes J^{\otimes Z}  \simeq_{KK^G} ((B/J) \otimes J)^{\otimes Z} \simeq_{KK^G} 0
\] 
in $KK^G$ where the last equivalence follows by Theorem \ref{thm-Izumi} combined with the $KK$-equivalence $(B/J) \otimes J \simeq_{KK} M_\fm \otimes J \simeq_{KK} 0$. It follows that ${j \otimes \mathrm{id}_{J^{\otimes Z}}}$ induces a $KK^G$-equivalence 
\[
 J\otimes J^{\otimes Z}  \simeq_{KK^G}  B \otimes J^{\otimes Z}   \simeq_{KK^G}  J^{\otimes Z}.
\]
Combining this with the $KK^G$-equivalence $J \otimes J^{\otimes Z} \simeq_{KK^G} J$, we see that $J^{\otimes Z}$ is $KK^G$-equivalent to $J$.

We have just proved the implication assuming $A$ is $KK$-equivalent to $\mathcal F_{1,\fm}$. Now suppose $A$ is $KK$-equivalent to ${M_\fn\oplus \mathcal F_{1,\fm}}$ for supernatural numbers $\fm,\fn$ of infinite type such that $\fm$ divides $\fn$. Then, we have
\[
M_\mathfrak n^{\otimes Y} \otimes \mathcal{F}_{1, \mathfrak m}^{\otimes Z}   \simeq_{KK^G}  M_\mathfrak n \otimes \mathcal{F}_{1, \mathfrak m}   \simeq_{KK^G}  0,
\]
for any finite group $G$, and for any finite $G$-sets $Y$, $Z$ by the previous part and Theorem \ref{thm-kk-equivalence}. By binomial expansion, the Benoulli-shift on $(M_\fn \oplus \mathcal{F}_{1, \fm} )^{\otimes Z}$ is isomorphic to
\[
\bigoplus_{[S] \in \mathrm{Sub}(Z)/G} \left(  \bigoplus_{F \in [S]}  M_\mathfrak n^{\otimes F} \otimes \mathcal{F}_{1, \mathfrak m}^{\otimes Z-F}  \right) \cong \bigoplus_{[S] \in \mathrm{Sub}(Z)/G}  \mathrm{Ind}^G_{G_S}( M_\mathfrak n^{\otimes S} \otimes \mathcal{F}_{1, \mathfrak m}^{\otimes Z-S}), 
\]
where $\mathrm{Sub}(Z)$ is the set of subsets of $Z$, equipped with the natural $G$-action induced from the $G$-action on $Z$. Each summand is $KK^G$-equivalent to zero unless $S= \emptyset$ or $Z$. It follows that the Bernoulli-shifts on $(M_\mathfrak n \oplus \mathcal F_{1,\mathfrak \mathfrak m})^{\otimes Z}$ is $KK^G$-equivalent $M_\mathfrak n^{\otimes Z}\oplus \mathcal F_{1,\mathfrak \mathfrak m}^{\otimes Z}$, which is $KK^G$-equivalent to ${M_\fn\oplus \mathcal F_{1,\fm}}$ by the previous part and Theorem \ref{thm-kk-equivalence}.

\eqref{item-C2-trivial} $\implies$ \eqref{item-K-absorbing}: 
If $A$ satisfies \eqref{item-C2-trivial}, the flip automorphism $\sigma_{A, A}$, as an element in $KK(A\otimes A, A\otimes A)$, is equal to the identity element $\mathrm{id}_{A\otimes A}$\footnote{To see this note that for any discrete group $G$, the universal property of $KK^G$ induces a functor $KK^G\to \mathrm{Fun}(BG,KK)$ where $BG$ is the category with one object and $G$ as automorphisms.}.
It remains to be shown that $K_*(A)\cong K_*(A\otimes A)$. 
By Theorem \ref{thm_aif}, $A$ is $KK$-equivalent to either $\mathcal F_{1,\fm}$ or $M_{\fn} \oplus \mathcal{F}_{1, \fm}$ for supernatural numbers $\fm, \fn$ such that $\fm$ is of infinite type, and $\fm$ divides $\fn$. 
By the K\"unneth theorem, it is enough to consider the latter case and show that $\fn$ is of essentially infinite type. 
As in the proof of \eqref{item-K-theory} $\implies$ \eqref{item-G-trivial}, we have a $KK^{C_2}$-equivalence ${(M_\fn\oplus \mathcal F_{1,\fm})^{\otimes C_2}}\simeq_{KK^{C_2}} {M_\fn ^{\otimes C_2}\oplus \mathcal F_{1,\fm}^{\otimes C_2}}$. 
Moreover, by the proof of \eqref{item-K-theory} $\implies$ \eqref{item-G-trivial}, $\mathcal F_{1,\fm}^{\otimes C_2}$ is $KK^{C_2}$-equivalent to $\mathcal F_{1,\fm}$ equipped with the trivial action. In particular, we have $K_0(\mathcal F_{1,\fm}^{\otimes C_2}\rtimes C_2)\cong K_0(\mathcal F_{1,\fm}\otimes C^*( C_2))=0$ and thus 
\[ K_0(M_\fn ^{\otimes C_2}\rtimes C_2)\cong K_0(A^{\otimes C_2}\rtimes C_2)\cong K_0(A^{\otimes 2}\otimes C^*(C_2))\cong K_0(M_\fn^{\otimes 2}\otimes C^*( C_2)).\]
It follows from Lemma \ref{lem_Kcompute} that $\fn$ must be of essentially infinite type. 
\end{proof}
The following corollary is a simple consequence of Theorem \ref{thm_GKK} and \cite[Corollary 6.4 (ii)]{Gabe2024}.
\begin{corollary}
	Let $G$ be a finite group, let $Y,Z$ be finite $G$-sets and let $\fm$ be a supernatural number of infinite type. Assume that each non-trivial element in $G$ acts non-trivially on $Y$ and $Z$.
	Then there is an equivariant isomorphism $\mathcal F_{1,\fm}^{\otimes Y} \cong \mathcal O_\infty^{\otimes Z}\otimes \mathcal F_{1,\fm}$ where $\mathcal F_{1,\fm}^{\otimes Y}$ and $\mathcal O_\infty^{\otimes Z}$ are equipped with the Bernoulli shift $G$-actions and $\mathcal F_{1,\fm}$ is equipped with the trivial $G$-action. 
\end{corollary}


\bibliography{Refs}

\newcommand{\etalchar}[1]{$^{#1}$}
\begin{thebibliography}{CEOO04}

\bibitem[Arv77]{Arveson}
William Arveson.
\newblock Notes on extensions of {$C^*$}-algebras.
\newblock {\em Duke Math. J.}, 44(2):329--355, 1977.
\newblock \href {https://doi.org/10.1215/s0012-7094-77-04414-3}
  {\path{doi:10.1215/s0012-7094-77-04414-3}}.

\bibitem[BCH94]{BCH}
Paul Baum, Alain Connes, and Nigel Higson.
\newblock Classifying space for proper actions and {$K$}-theory of group
  {$C^\ast$}-algebras.
\newblock In {\em {$C^\ast$}-algebras: 1943--1993 ({S}an {A}ntonio, {TX},
  1993)}, volume 167 of {\em Contemp. Math.}, pages 240--291. Amer. Math. Soc.,
  Providence, RI, 1994.
\newblock \href {https://doi.org/10.1090/conm/167/1292018}
  {\path{doi:10.1090/conm/167/1292018}}.

\bibitem[BEL23]{Bunke2023a}
Ulrich Bunke, Alexander Engel, and Markus Land.
\newblock {A stable $\infty$-category for equivariant $KK$-theory}, 2023.
\newblock \href {https://arxiv.org/abs/2102.13372} {\path{arXiv:2102.13372}}.

\bibitem[Bun23]{Bunke2023}
Ulrich Bunke.
\newblock {$K$-theory of crossed products via homotopy theory}, 2023.
\newblock \href {https://arxiv.org/abs/2311.06562} {\path{arXiv:2311.06562}}.

\bibitem[CEKN24]{CEKN}
Sayan Chakraborty, Siegfried Echterhoff, Julian Kranz, and Shintaro Nishikawa.
\newblock {K}-theory of noncommutative {B}ernoulli shifts.
\newblock {\em Math. Ann.}, 388(3):2671--2703, 2024.
\newblock \href {https://doi.org/10.1007/s00208-023-02587-w}
  {\path{doi:10.1007/s00208-023-02587-w}}.

\bibitem[CEL13]{CEL}
Joachim Cuntz, Siegfried Echterhoff, and Xin Li.
\newblock On the {$K$}-theory of crossed products by automorphic semigroup
  actions.
\newblock {\em Q. J. Math.}, 64(3):747--784, 2013.
\newblock \href {https://doi.org/10.1093/qmath/hat021}
  {\path{doi:10.1093/qmath/hat021}}.

\bibitem[CEOO04]{CEO}
J\'{e}r\^{o}me Chabert, Siegfried Echterhoff, and Herv\'e Oyono-Oyono.
\newblock Going-down functors, the {K}\"{u}nneth formula, and the
  {B}aum-{C}onnes conjecture.
\newblock {\em Geom. Funct. Anal.}, 14(3):491--528, 2004.
\newblock \href {https://doi.org/10.1007/s00039-004-0467-6}
  {\path{doi:10.1007/s00039-004-0467-6}}.

\bibitem[CET{\etalchar{+}}21]{CETWW}
Jorge Castillejos, Samuel Evington, Aaron Tikuisis, Stuart White, and Wilhelm
  Winter.
\newblock Nuclear dimension of simple {$\rm C^*$}-algebras.
\newblock {\em Invent. Math.}, 224(1):245--290, 2021.
\newblock \href {https://doi.org/10.1007/s00222-020-01013-1}
  {\path{doi:10.1007/s00222-020-01013-1}}.

\bibitem[CGS{\etalchar{+}}23]{Carrion2023}
Jos{\'e}~R. Carri{\'o}n, James Gabe, Christopher Schafhauser, Aaron Tikuisis,
  and Stuart White.
\newblock Classifying $^*$-homomorphisms {I}: Unital simple nuclear
  ${C}^*$-algebras, 2023.
\newblock \href {https://arxiv.org/abs/2307.06480} {\path{arXiv:2307.06480}}.

\bibitem[CS75]{CS}
Alain Connes and Erling St{\o}rmer.
\newblock Entropy for automorphisms of {$II_{1}$} von {N}eumann algebras.
\newblock {\em Acta Math.}, 134(3-4):289--306, 1975.
\newblock \href {https://doi.org/10.1007/BF02392105}
  {\path{doi:10.1007/BF02392105}}.

\bibitem[CS86]{CS86}
Joachim Cuntz and Georges Skandalis.
\newblock Mapping cones and exact sequences in {$KK$}-theory.
\newblock {\em J. Operator Theory}, 15(1):163--180, 1986.
\newblock URL:
  \url{https://www.theta.ro/jot/archive/1986-015-001/1986-015-001-007.pdf}.

\bibitem[EGLN15]{EGLN}
George~A. Elliott, Guihua Gong, Huaxin Lin, and Zhuang Niu.
\newblock On the classification of simple amenable {$C^*$}-algebras with finite
  decomposition rank, {II}, 2015.
\newblock \href {https://arxiv.org/abs/1507.03437} {\path{arXiv:1507.03437}}.

\bibitem[ER78]{ER}
Edward~G. Effros and Jonathan Rosenberg.
\newblock {$C^*$}-algebras with approximately inner flip.
\newblock {\em Pacific J. Math.}, 77(2):417--443, 1978.
\newblock \href {https://doi.org/10.2140/pjm.1978.77.417}
  {\path{doi:10.2140/pjm.1978.77.417}}.

\bibitem[EST24]{TikuisisCorrigendum}
Dominic Enders, Andr{\'e} Schemaitat, and Aaron Tikuisis.
\newblock {Corrigendum to ``{$K$}-theoretic Characterization of
  {$C^*$}-algebras with Approximately Inner Flip''}.
\newblock {\em Int. Math. Res. Not. IMRN}, 2024(9):7680--7699, 2024.
\newblock \href {https://doi.org/10.1093/imrn/rnad295}
  {\path{doi:10.1093/imrn/rnad295}}.

\bibitem[GL21]{GardellaLupini}
Eusebio Gardella and Martino Lupini.
\newblock Group amenability and actions on {$\mathcal{Z}$}-stable {$
  C^*$}-algebras.
\newblock {\em Adv. Math.}, 389:Paper No. 107931, 33, 2021.
\newblock \href {https://doi.org/10.1016/j.aim.2021.107931}
  {\path{doi:10.1016/j.aim.2021.107931}}.

\bibitem[GS24]{Gabe2024}
James Gabe and G{\'a}bor Szab{\'o}.
\newblock {The dynamical Kirchberg--Phillips theorem}.
\newblock {\em Acta Math}, 232:1--77, 2024.
\newblock \href {https://doi.org/10.4310/ACTA.2024.v232.n1.a1}
  {\path{doi:10.4310/ACTA.2024.v232.n1.a1}}.

\bibitem[HK01]{HK}
Nigel Higson and Gennadi Kasparov.
\newblock {$E$}-theory and {$KK$}-theory for groups which act properly and
  isometrically on {H}ilbert space.
\newblock {\em Invent. Math.}, 144(1):23--74, 2001.
\newblock \href {https://doi.org/10.1007/s002220000118}
  {\path{doi:10.1007/s002220000118}}.

\bibitem[HW08]{HirshbergWinter}
Ilan Hirshberg and Wilhelm Winter.
\newblock Permutations of strongly self-absorbing {$C^*$}-algebras.
\newblock {\em Internat. J. Math.}, 19(9):1137--1145, 2008.
\newblock \href {https://doi.org/10.1142/S0129167X08005011}
  {\path{doi:10.1142/S0129167X08005011}}.

\bibitem[Izu04]{Izumi2}
Masaki Izumi.
\newblock Finite group actions on {$C^*$}-algebras with the {R}ohlin property.
  {I}.
\newblock {\em Duke Math. J.}, 122(2):233--280, 2004.
\newblock \href {https://doi.org/10.1215/S0012-7094-04-12221-3}
  {\path{doi:10.1215/S0012-7094-04-12221-3}}.

\bibitem[Izu19]{Izumi}
Masaki Izumi.
\newblock The {$K$}-theory of the flip automorphisms.
\newblock In {\em Operator algebras and mathematical physics}, volume~80 of
  {\em Adv. Stud. Pure Math.}, pages 123--137. Math. Soc. Japan, Tokyo, 2019.
\newblock \href {https://doi.org/10.2969/aspm/08010123}
  {\path{doi:10.2969/aspm/08010123}}.

\bibitem[Kas88]{Kasparov}
Gennadi~G. Kasparov.
\newblock Equivariant {$KK$}-theory and the {N}ovikov conjecture.
\newblock {\em Invent. Math.}, 91(1):147--201, 1988.
\newblock \href {https://doi.org/10.1007/BF01404917}
  {\path{doi:10.1007/BF01404917}}.

\bibitem[Kis81]{Kishimoto}
Akitaka Kishimoto.
\newblock Outer automorphisms and reduced crossed products of simple
  {$C\sp{\ast} $}-algebras.
\newblock {\em Comm. Math. Phys.}, 81(3):429--435, 1981.
\newblock \href {https://doi.org/10.1007/BF01209077}
  {\path{doi:10.1007/BF01209077}}.

\bibitem[KN24]{Kranz2024}
Julian Kranz and Shintaro Nishikawa.
\newblock {Bernoulli shifts on additive categories and algebraic $ K $-theory
  of wreath products}.
\newblock {\em to appear in Algebr. Geom. Topol.}, 2024.
\newblock \href {https://arxiv.org/abs/2401.14806} {\path{arXiv:2401.14806}}.

\bibitem[Laf12]{Lafforgue}
Vincent Lafforgue.
\newblock La conjecture de {B}aum-{C}onnes \`a coefficients pour les groupes
  hyperboliques.
\newblock {\em J. Noncommut. Geom.}, 6(1):1--197, 2012.
\newblock \href {https://doi.org/10.4171/JNCG/89} {\path{doi:10.4171/JNCG/89}}.

\bibitem[Li19]{XinLi}
Xin Li.
\newblock {$K$}-theory for generalized {L}amplighter groups.
\newblock {\em Proc. Amer. Math. Soc.}, 147(10):4371--4378, 2019.
\newblock \href {https://doi.org/10.1090/proc/14619}
  {\path{doi:10.1090/proc/14619}}.

\bibitem[Li22]{XinLiPartial}
Xin Li.
\newblock K-theory for semigroup {$\rm C^*$}-algebras and partial crossed
  products.
\newblock {\em Comm. Math. Phys.}, 390(1):1--32, 2022.
\newblock \href {https://doi.org/10.1007/s00220-021-04194-9}
  {\path{doi:10.1007/s00220-021-04194-9}}.

\bibitem[MN06]{MN}
Ralf Meyer and Ryszard Nest.
\newblock The {B}aum-{C}onnes conjecture via localisation of categories.
\newblock {\em Topology}, 45:209--259, 2006.
\newblock \href {https://doi.org/10.1016/j.top.2005.07.001}
  {\path{doi:10.1016/j.top.2005.07.001}}.

\bibitem[MS14]{matui2014ƶ}
Hiroki Matui and Yasuhiko Sato.
\newblock {$\mathcal Z$-stability of crossed products by strongly outer actions
  II}.
\newblock {\em Amer. J. Math.}, 136(6):1441--1496, 2014.
\newblock \href {https://doi.org/10.1353/ajm.2014.0043}
  {\path{doi:10.1353/ajm.2014.0043}}.

\bibitem[Phi87]{Phillips}
N.~Christopher Phillips.
\newblock {\em Equivariant {$K$}-theory and freeness of group actions on
  {$C^*$}-algebras}, volume 1274 of {\em Lecture Notes in Mathematics}.
\newblock Springer-Verlag, Berlin, 1987.
\newblock \href {https://doi.org/10.1007/BFb0078657}
  {\path{doi:10.1007/BFb0078657}}.

\bibitem[Phi00]{KirchbergPhillips}
N.~Christopher Phillips.
\newblock A classification theorem for nuclear purely infinite simple
  {$C^*$}-algebras.
\newblock {\em Doc. Math.}, 5:49--114, 2000.
\newblock \href {https://doi.org/10.4171/DM/75} {\path{doi:10.4171/DM/75}}.

\bibitem[Pop06]{Popa}
Sorin Popa.
\newblock Some rigidity results for non-commutative {B}ernoulli shifts.
\newblock {\em J. Funct. Anal.}, 230(2):273--328, 2006.
\newblock \href {https://doi.org/10.1016/j.jfa.2005.06.017}
  {\path{doi:10.1016/j.jfa.2005.06.017}}.

\bibitem[RS87]{Rosenberg1987}
Jonathan Rosenberg and Claude Schochet.
\newblock {The K{\"u}nneth theorem and the universal coefficient theorem for
  Kasparov’s generalized K-functor}.
\newblock {\em Duke Math. J.}, 55(2):431, 1987.
\newblock \href {https://doi.org/10.1215/s0012-7094-87-05524-4}
  {\path{doi:10.1215/s0012-7094-87-05524-4}}.

\bibitem[Ser77]{Serre}
Jean-Pierre Serre.
\newblock {\em Linear representations of finite groups}.
\newblock Graduate Texts in Mathematics, Vol. 42. Springer-Verlag, New
  York-Heidelberg, 1977.
\newblock Translated from the second French edition by Leonard L. Scott.
\newblock \href {https://doi.org/10.1007/978-1-4684-9458-7}
  {\path{doi:10.1007/978-1-4684-9458-7}}.

\bibitem[Sza18a]{Szabo}
G\'{a}bor Szab\'{o}.
\newblock Equivariant {K}irchberg-{P}hillips-type absorption for amenable group
  actions.
\newblock {\em Comm. Math. Phys.}, 361(3):1115--1154, 2018.
\newblock \href {https://doi.org/10.1007/s00220-018-3110-3}
  {\path{doi:10.1007/s00220-018-3110-3}}.

\bibitem[Sza18b]{Szabo2018}
G{\'a}bor Szab{\'o}.
\newblock {Strongly self-absorbing $C^*$-dynamical systems}.
\newblock {\em Trans. Amer. Math. Soc.}, 370(1):99--130, 2018.
\newblock \href {https://doi.org/10.1090/tran/6931}
  {\path{doi:10.1090/tran/6931}}.

\bibitem[Sza19]{Szabo2019}
G{\'a}bor Szab{\'o}.
\newblock {Actions of certain torsion-free elementary amenable groups on
  strongly self-absorbing $C^*$-algebras}.
\newblock {\em Comm. Math. Phys.}, 371:267--284, 2019.
\newblock \href {https://doi.org/10.1007/s00220-019-03435-2}
  {\path{doi:10.1007/s00220-019-03435-2}}.

\bibitem[Tik16]{Tikuisis}
Aaron Tikuisis.
\newblock K-theoretic characterization of {C}*-algebras with approximately
  inner flip.
\newblock {\em Int. Math. Res. Not. IMRN}, (18):5670--5694, 2016.
\newblock \href {https://doi.org/10.1093/imrn/rnv334}
  {\path{doi:10.1093/imrn/rnv334}}.

\bibitem[TW07]{TW}
Andrew~S. Toms and Wilhelm Winter.
\newblock Strongly self-absorbing {$C^*$}-algebras.
\newblock {\em Trans. Amer. Math. Soc.}, 359(8):3999--4029, 2007.
\newblock \href {https://doi.org/10.1090/S0002-9947-07-04173-6}
  {\path{doi:10.1090/S0002-9947-07-04173-6}}.

\bibitem[TWW17]{TWW}
Aaron Tikuisis, Stuart White, and Wilhelm Winter.
\newblock Quasidiagonality of nuclear {$C^\ast$}-algebras.
\newblock {\em Ann. of Math. (2)}, 185(1):229--284, 2017.
\newblock \href {https://doi.org/10.4007/annals.2017.185.1.4}
  {\path{doi:10.4007/annals.2017.185.1.4}}.

\bibitem[Voi95]{Voiculescu1995}
Dan Voiculescu.
\newblock Dynamical approximation entropies and topological entropy in operator
  algebras.
\newblock {\em Comm. Math. Phys.}, 170(2):249--281, 1995.
\newblock \href {https://doi.org/10.1007/BF02108329}
  {\path{doi:10.1007/BF02108329}}.

\bibitem[Whi23]{White2023}
Stuart White.
\newblock Abstract classification theorems for amenable ${C}^*$-algebras.
\newblock In {\em Proceedings of the ICM, 2022}, volume~4, pages 3314--3338.
  EMS Press, 2023.
\newblock \href {https://doi.org/10.4171/ICM2022/183}
  {\path{doi:10.4171/ICM2022/183}}.

\bibitem[Win11]{Winter}
Wilhelm Winter.
\newblock Strongly self-absorbing {$C^*$}-algebras are {$\mathcal Z$}-stable.
\newblock {\em J. Noncommut. Geom.}, 5(2):253--264, 2011.
\newblock \href {https://doi.org/10.4171/JNCG/74} {\path{doi:10.4171/JNCG/74}}.

\bibitem[Win18]{WinterICM}
Wilhelm Winter.
\newblock Structure of nuclear {$\rm C^*$}-algebras: from quasidiagonality to
  classification and back again.
\newblock In {\em Proceedings of the {I}nternational {C}ongress of
  {M}athematicians---{R}io de {J}aneiro 2018. {V}ol. {III}. {I}nvited
  lectures}, pages 1801--1823. World Sci. Publ., Hackensack, NJ, 2018.
\newblock \href {https://doi.org/10.1142/9789813272880_0118}
  {\path{doi:10.1142/9789813272880_0118}}.

\end{thebibliography}
\bibliographystyle{alphaurl}

\end{document}